\documentclass[10pt,notitlepage,twoside,a4paper]{amsart}
 \usepackage{amsfonts}

\usepackage{amsmath,amssymb,enumerate}

\usepackage{epsfig,fancyhdr,color}

\usepackage{amssymb}
\usepackage{amsmath,amsthm}
\usepackage{latexsym}
\usepackage{amscd}
\usepackage{psfrag}
\usepackage{graphicx}
\usepackage[latin1]{inputenc}
\usepackage[all]{xy}
\usepackage[mathcal]{eucal}

\definecolor{NoteColor}{rgb}{1,0,0}

\renewcommand{\textsc}{\textcolor{red}}

\newtheorem{theorem}{\rm\bf Theorem}[section]
\newtheorem{proposition}[theorem]{\rm\bf Proposition}
\newtheorem{lemma}[theorem]{\rm\bf Lemma}

\newtheorem*{theorem 1}{\rm\bf Proposition 1}
\newtheorem*{theorem 2}{\rm\bf Proposition 2}

\theoremstyle{definition}
\newtheorem{definition}[theorem]{\rm\bf Definition}

\theoremstyle{remark}
\newtheorem{remark}[theorem]{\rm\bf Remark}
\newtheorem{remarks}[theorem]{\rm\bf Remarks}
\newtheorem{example}[theorem]{\rm\bf Example}

\newtheorem{question}[theorem]{\rm\bf Question}

\def\interieur#1{\mathord{\mathop{\kern 0pt #1}\limits^\circ}}

\title[Teichm\"uller spaces]{On various Teichm\"uller spaces of a surface of infinite topological type}

\author{D. Alessandrini}
\address{Daniele Alessandrini,  Institut de Recherche Math{\'e}matique Avanc\'ee,
Universit{\'e} de Strasbourg and CNRS,
7 rue Ren\'e Descartes,
 67084 Strasbourg Cedex, France}
\email{alessand@math.u-strasbg.fr}

\author{L. Liu}
\address{Lixin Liu, Department of Mathematics, Sun Yat-sen (Zhongshan) University, 510275, Guangzhou, P. R. China}
\email{mcsllx@mail.sysu.edu.cn}

\author{A. Papadopoulos}
\address{Athanase Papadopoulos, Institut de Recherche Math{\'e}matique Avanc\'ee,
Universit{\'e} de Strasbourg and CNRS,
7 rue Ren\'e Descartes,
 67084 Strasbourg Cedex, France} \email{athanase.papadopoulos@math.unistra.fr}
\date{\today}

\author{W. Su}
\address{Weixu Su, Department of Mathematics, Sun Yat-sen(Zhongshan) University, 510275, Guangzhou, P. R. China}
\email{su023411040@163.com}

\begin{document}

\begin{abstract} We show that the length spectrum metric on Teichm\"uller spaces of surfaces of infinite topological type is complete. We also give related results and examples that compare the length spectrum Teichm\'uller space with quasiconformal and the Fenchel-Nielsen Teichm\"uller spaces on such surfaces. 
\end{abstract} 

\maketitle


\noindent AMS Mathematics Subject Classification:   32G15 ; 30F30 ; 30F60.
\medskip

\noindent Keywords:  Surfaces of infinite topological type,  Teichm\"uller space, Teichm\"uller metric, quasiconformal metric,  length spectrum metric,  Fenchel-Nielsen coordinates, Fenchel-Nielsen metric.
\medskip

\noindent L. Liu and W. Su are partially supported by NSFC grants 10871211 and 11011130207.

\tableofcontents
  
\section{Introduction}\label{intro}

In this paper, we investigate various Teichm\"uller spaces associated to a surface of  infinite topological type, in continuation of works that were done in  \cite{LP} and \cite{ALPSS}. An initial impulse to these works was given by a paper by H. Shiga  \cite{Shiga}.

Let $S$ be an orientable connected surface of infinite topological type. More precisely, we assume that $S$ is obtained by gluing along their boundary components a countably infinite number of generalized pairs of pants. Here, a generalized pair of pants is a sphere with three holes, a hole being either a point removed (leaving a puncture of the pair of pants) or an open disk removed (leaving a boundary component of the pair of pants).

We study Teichm\"uller spaces of $S$. We recall that unlike in the case of surfaces of  finite type, there are several Teichm\"uller spaces associated to the surface $S$,  each of of these Teichm\"uller spaces heavily depending (even set-theoretically) on the choice of a basepoint for that space. Furthermore, even if we fix a basepoint, the Teichm\"uller space depends (again, set-theoretically) on a distance function that we put on that space. There are various distances that one can use here. For instance, such a distance between two hyperbolic structures can measure suprema of ratios of geodesic lengths of simple closed curves with respect to these two structures, or best quasiconformal homeomorphism constants between them, or best bi-Lipschitz homeomorphism constants, or it can be equal to a sup norm associated to Fenchel-Nielsen coordinates, and so on. We refer to the papers \cite{LP} and \cite{ALPSS} for a discussion of such  ideas.
The Teichm\"uller spaces we obtain have consequently different names, and
in this paper, we shall deal with the so-called ``length-spectrum Teichm\"uller space",    ``quasiconformal Teichm\"uller space" and ``Fenchel-Nielsen Teichm\"uller space".

For the purpose of stating the results, let us briefly review the definitions.

The \emph{length-spectrum Teichm\"uller space},  $\mathcal{T}_{ls}(H_0)$, with basepoint (the homotopy class of) a hyperbolic structure $H_0$ on $S$,
is the space of homotopy classes of hyperbolic structures $H$ on $S$ such that the ratios of lengths of simple closed geodesics measured in the metric $H_0$ and $H$ are uniformly bounded (see more precisely Definition \ref{def:ls} below). This space is equipped with a natural distance $d_{ls}$ called the \emph{length-spectrum distance}, obtained by  taking the logarithm of the supremum of ratios of geodesic  lengths between (homotopy classes of) hyperbolic structures; see Formula (\ref{eq:d-ls}) below.

The \emph{quasiconformal Teichm\"uller space} with basepoint (the homotopy class of) a conformal structure $H_0$ on $S$,  $\mathcal{T}_{qc}(H_0)$,  is the space of homotopy classes of conformal structures $H$ on $S$ such that there exists a quasiconformal mapping homotopic to the identity between the structures $H_0$ and $H$ (see more precisely Definition \ref{def:qc} below). We denote this space by $\mathcal{T}_{qc}(H_0)$. This space is equipped with a natural distance $d_{qc}$, the \emph{quasiconformal} or \emph{Teichm\"uller distance}, given by the logarithm of the dilatation of the best quasiconformal map homotopic to the identity between the two structures; see Formula (\ref{eq:qc}) below.

A simple closed curves on $S$ is said to be essential if it is not homotopic to a point or to a puncture (but it can be homotopic to a boundary component). 
We let $\mathcal{S}=\mathcal{S}(S)$ be set of isotopy classes of essential simple closed curves on $S$. 
Given an element $\alpha$ of $\mathcal{S}$ and a (homotopy class of) hyperbolic structure $H$ on $S$, we denote by $l_H(\alpha)$ the length of the unique closed $H$-geodesic in the class $\alpha$. 

By an abuse of language, we shall often identify a hyperbolic metric (or conformal structure) on $S$ with the homotopy class of that metric (or conformal structure) as an element of Teichm\"uller space.

A basic result that is used in comparing the two Teichm\"uller spaces $(\mathcal{T}_{ls}, d_{ls})$ and  $(\mathcal{T}_{qc}, d_{qc })$ is a result of Wolpert saying that if $H$ and $H'$ are two hyperbolic structures on the surface $S$, then, for any $K$-quasiconformal map $f:(S,H)\to (S,H')$ and for any element $\alpha$ in $\mathcal{S}(S)$, we have the following inequality:
\begin{equation}\label{eq:Wolpert}
\frac{1}{K}\leq\frac{l_{H'}(f(\alpha))}{ l_{H}(\alpha)}\leq K.
\end{equation} 
For a proof, see \cite{Abikoff}. We refer to this inequality as \emph{Wolpert's inequality}.

From this inequality, we obtain a natural inclusion map
\begin{equation}\label{eq:inclusion2}
\mathcal{T}_{qc}(H_0)\hookrightarrow \mathcal{T}_{ls}(H_0).
\end{equation}

    In general, this inclusion map is not surjective (see \cite{LP} for an example), bur it is continuous (and Lipschitz), since Wolpert's inequality also implies that for any two elements $H$ and $H'$ in $\mathcal{T}_{qc}(S_0)$, we have
\begin{equation}\label{eq:lsqc}
d_{ls}(H,H')\leq d_{qc}(H,H').
\end{equation}
 
 Given an element $\gamma$ in $\mathcal{S}$ and a hyperbolic metric $H$ on $S$, we denote by $l_H(\gamma)$ the length of the unique geodesic in the class $\gamma$, for the structure $H$.

We shall also use Fenchel-Nielsen coordinates for hyperbolic structures. The coordinates are defined relative to a pair of pants decomposition. The notion of hyperbolic pair of pants decomposition of $S$ has to be used with some special care, one reason being that unlike the case of surfaces of finite type, if we are given a topological pair of pants decomposition $\mathcal{P}=\{C_i\}_{i \in I}$ of $S$ and a hyperbolic structure $H_0$ on $S$, and if we replace each simple closed curve $C_i$ by the $H_0$-geodesic in its homotopy class, then some of these closed geodesics might accumulate on a geodesic of infinite length, and then the union of the closed geodesics might not be a geodesic pair of pants decomposition. Such a phenomenon can be seen in examples of Basmajian, in his paper \cite{Basmajian}. In the paper \cite{ALPSS}, we gave a necessary and sufficient condition (which we called \emph{Nielsen-convexity}) under which given a hyperbolic structure on a surface of infinite type, a topological pair of pants decomposition (or, equivalently, \emph{any} topological pair of pants decomposition) can be made geodesic. Consequently, when we shall talk about Fenchel-Nielsen coordinates for a hyperbolic surface, we shall tacitly assume that the underlying hyperbolic structure is Nielsen-convex.

In the paper \cite{ALPSS}, we also introduced the notion of a \emph{Fenchel-Nielsen Teichm\"uller space}, $\mathcal{T}_{FN}(H_0)$, based at a hyperbolic surface $H_0$, with its associated \emph{Fenchel-Nielsen metric} $d_{FN}$, relative to a fixed geodesic pants decomposition $\mathcal{P}$ of $H_0$.

          Given a pair of pants decomposition $\mathcal{P}=\{C_i\}_{i=\in I}$ (where $I$ is an infinite countable set) of the surface $S$, the following condition was formulated by Shiga in his paper \cite{Shiga}:
          
          \begin{equation}
  \label{LM} \exists M>0,\forall i\in I, \frac{1}{M}\leq l_H(C_i)\leq M.
 \end{equation}
              
We shall say that a hyperbolic structure $H$ satisfying (\ref{LM}) satisfies \emph{Shiga's condition} with respect to the pair of pants decomposition $\mathcal{P}=\{C_i\}$  $i\in I$.

In \cite{LP} (Theorem 4.14), we proved that if the base hyperbolic metric $H_0$ satisfies Shiga's Condition,
    then we have 
     $\mathcal{T}_{qc}(H_0)=\mathcal{T}_{ls}(H_0)$ (set-theoretically).

If the base topological surface were of finite type, then it is known that the length-spectrum and the quasiconformal Teichm\"uller spaces coincide setwise, and that the topologies defined on that set by the length-spectrum metric and the quasiconformal metrics are the same. This is because the Teichm\"uller space topology is induced from the embedding of that space in the space $\mathbb{R}_+^{\mathcal{S}}$ of positive functions on $\mathcal{S}$, equipped with the weak topology via the length functions.
The fact that the topology induced by the length-spectrum metric coincides with this topology follows from the fact that the geodesic length functions of some finite number of elements of $\mathcal{S}$ are sufficient to parametrize Teichm\"uller space and to define its topology, see \cite{FLP}. See also \cite{Li} and \cite{Liu}.

The case of surfaces of infinite type is different. The first negative result in this direction is a result by Shiga, who proved in \cite{Shiga} (Theorem 1.1) that
there exists a hyperbolic structure $H_0$ on a surface of infinite type and a
sequence $(H_n)$, $n\geq 1$ of hyperbolic structures in 
 $\mathcal{T}_{ls}(H_0)$ which (when they are regarded as conformal structures) are at the same time are in $\mathcal{T}_{qc}(H_0)$ that satisfy
$$d_{ls}(H_n,H_0)\to 0, \ \mathrm{while} \ d_{qc}(H_n,H_0)\to \infty.$$
This shows that  $d_{ls}$ does not induce the same topology as that of $d_{qc}$ on
$\mathcal{T}_{qc}(H_0)$.

In the same paper, Shiga showed that if the hyperbolic metric $H_0$ satisfies Property (\ref{LM}), then
 $d_{ls}$ induces the same topology as that of $d_{qc}$ on $\mathcal{T}_{qc}(S)$.

Furthermore, Shiga showed that there exists a Riemann surface of infinite type such that the length spectrum distance $d_{ls}$ restricted to the quasiconformal Teichm\"uller space $\mathcal{T}_{qc}(S_0)$ is not complete (\cite[Corollary 1.1]{Shiga}). 
We shall give below (Example \ref{example:non-lower}) another example of this phenomenon, by a construction that is probably simpler than the one of Shiga. The hyperbolic structure in this example is also different from the one given by Shiga, because in our example the surface (as a metric space) is complete whereas in Shiga's example it is not.

   We prove below (Proposition \ref{prop:non-inc}) that for some base hyperbolic structures $H_0$, we have  $\mathcal{T}_{ls}(H_0)\not\subset\mathcal{T}_{FN}(H_0)$. 
  We also give an example of a hyperbolic structure $H_0$ and a sequence of points $x_i, i=1, \cdots$  in $\mathcal{T}_{ls}(H_0)\cap \mathcal{T}_{FN}(H_0)$ such that $\lim_{n\to\infty}d_{ls}(x_n,
H_0)=0$ while $\lim_{n\to\infty}d_{FN}(x_n, H_0)=\infty$ (Proposition \ref{prop:non-inc2}).

The length spectrum metric on any Teichm\"uller space of a conformally finite type  Riemann surface  is complete (see \cite[Theorem 2.25]{LP}). The proof in \cite{LP} does not extend to the case of Teichm\"uller spaces of  surfaces of infinite topological type. We prove this result for surfaces of infinite topological type in \S \ref{s:completeness} below. More precisely, we prove that  for any base hyperbolic metrics $H_0$ on $S$, the metric space $(\mathcal{T}_{ls}(H_0), d_{ls})$ is complete (Theorem \ref{Theorem:complete}). This result answers a question we raised in \cite{LP} (Question 2.22).

\section{The length spectrum and the quasiconformal Teichm\"uller spaces}\label{qc-ls}

<for thr reader's convenience, we review briefly a few basic facts about the length spectrum and the quasiconformal Teichm\"uller spaces. 

All the homotopies of a surface that we consider in this paper  preserve the punctures and preserve setwise the boundary components at all times.

       Throughout this section,  $H_0$  is a fixed hyperbolic surface on the surface $S$, called the base hyperbolic structure. 
       Given  a hyperbolic surface $H$ on $S$ and a homeomorphism $f:(S,H_0)\to (S,H)$, we define the  {\it length-spectrum constant of $f$} to be the quantity 
       
 \begin{equation}\label{eq:ls}
L(f)= \sup_{\alpha\in\mathcal{S}(H)} \left\{ \frac{l_{H'}(f(\alpha))}{l_{H}(\alpha)},\frac{l_{H}(\alpha)}{l_{H'}(f(\alpha))}\right\}.
\end{equation}
 
 This quantity depends only on the homotopy class of $f$.
 
     We  say that $f$ is {\it length-spectrum bounded} if $L(f)<\infty$.

In the setting of the length spectrum Teichm\"uller space, we consider the collection of hyperbolic structures $H$ on $S$ such that the identity map $\mathrm{Id}: (S,H_0)\to (S,H)$ is length-spectrum bounded. Given two such hyperbolic structures $H$ and $H'$, we write $H\sim H'$ if there exists an isometry (or, equivalently, a length spectrum preserving homeomorphism)from $f(S,H)$ to $(S',H')$ which is homotopic to the identity.  
The relation $\sim$ is an equivalence relation on the set of length-spectrum bounded hyperbolic structures $H$ with respect to the basepoint $H_0$.

\begin{definition}\label{def:ls} The {\it length-spectrum Teichm\"uller space} $\mathcal{T}_{ls}(H_0)$ is the space of $\sim$-equivalence classes of length-spectrum bounded  hyperbolic structures. The {\it basepoint} of this Teichm\"uller space is the equivalence class $H_0$.
\end{definition}

 We note that the fact that we do not ask our homotopies to preserve pointwise the boundary of the surface corresponds to working with what is usually called the {\it reduced} Teichm\"uller space, instead of Teichm\"uller space. (In the latter case, the homotopies that define the equivalence relation are required to induce the identity map on each boundary component.) Since all the Teichm\"uller spaces that we use in this paper are reduced, we shall use, for simplicity, the terminology {\it Teichm\"uller space} instead of {\it reduced Teichm\"uller space}.

 The topology of $\mathcal{T}_{ls}(H_0)$ is induced by the {\it length-spectrum} metric $d_{ls}$, defined by taking the distance $d_{ls}(H,H')$ between two points in $\mathcal{T}_{ls}(H_0)$ represented by two hyperbolic surfaces $H$ and $H'$ to be
        \begin{equation}\label{eq:d-ls}
        d_{ls}(H,H')=\frac{1}{2}\log L(f'\circ  f^{-1}).
        \end{equation}
        
        (It may be useful to recall here that the length-spectrum constant  of a length-spectrum bounded homeomorphism only depends on the homotopy class of such a homeomorphism.)

The fact that the function $d_{ls}$ satisfies the properties of a metric is straightforward, except perhaps for the separation axiom, see
\cite{LP}.

A \emph{Riemann surface} is a one-dimensional complex manifold.

Given a real number $K\geq 1$, a homeomorphism $f:R\to R'$ between two Riemann surfaces is said to be \emph{$K$-quasiconformal} if $f$ has locally distributional derivatives satisfying at each point the following inequality:
\[\vert f_{\overline{z}}\vert \leq \frac{K-1}{K+1} \vert f_z\vert.\]

The {\it quasiconformal dilatation}, or, in short, the {\it dilatation} of $f$, is the infimum of the real numbers $K$ for which $f$ is $K$-quasiconformal.
         
           In the setting of the quasiconformal Teichm\"uller space with basepoint a Riemann structure surface $R_0$ on $S$, we only consider Riemann surfaces $R$  on $S$ such that the identity map $\mathrm{Id}:(S,R_0)\to (S,R)$ is quasiconformal. Given two such conformal structures $R$ and $R'$, we write $R\sim R'$ if there exists a conformal map from $(S,R)$ to $(S',R')$ which is homotopic to the identity.  
The relation $\sim$ is an equivalence relation on the set of conformal structures $R$ on $S$, with respect to the basepoint $R_0$.

       \begin{definition}\label{def:qc}
        Consider a Riemann surface structure $R_0$ on $X$. Its {\it  quasiconformal Teichm\"uller space}, $\mathcal{T}_{qc}(R_0)$, is the set of $\sim$-equivalence classes of Riemann surface structures on $S$.        
                \end{definition}

         The space $\mathcal{T}_{qc}(R_0)$ is equipped with the  \emph{quasiconformal metric}, also called the \emph{Teichm\"uller metric}, of wxhich we also recall the definition: Given two  
(equivalence classes of)  Riemann surface structures $R$ and $R'$ on $S$,  their  {\it quasiconformal distance} $d_{qc}(R,R')$ is defined as
        \begin{equation}\label{eq:qc} d_{qc}(R,R')=\frac{1}{2}\log \inf \{K(f)\}
        \end{equation}
        where the infimum is taken over  quasiconformal dilatations $K(f)$ of homeomorphisms $f:(S,R)\to (S,R')$ which are homotopic to the identity.

        The equivalence class of the marked Riemann surface $R_0$ is the {\it basepoint} of $\mathcal{T}_{qc}(R_0)$.

        We refer to Nag \cite{Nag} for an exposition of the quasiconformal theory of infinite-dimensional Teichm\"uller spaces.  In particular, it is known that the quasiconformal metric is complete.

Douady and Earle gave in  \cite{DE} a proof of the fact that any quasiconformal Teichm\"uller space $\mathcal{T}_{qc}(R_0)$ is contractible (see \cite[ Theorem 3]{DE}, where this result is also attributed to Tukia).  
It is unknown whether the length spectrum Teichm\"uller spaces are contractible.

\section{The Fenchel-Nielsen Teichm\"uller spaces}

We shall consider Fenchel-Nielsen coordinates for spaces of homotopy classes of hyperbolic structures on $S$. 
We carried out in \cite{ALPSS} a study of these parameters in the setting of surfaces of infinite type.
These parameters are associated to a fixed base hyperbolic structure equipped with a fixed geodesic pair of pants decomposition $\mathcal{P}=\{C_i\}_{i\in I}$. 
The boundary components of $S$ (if they exist) are all homeomorphic to circles, and are part of the curve system $C_i$ in the pair of pants decomposition. Fenchel-Nielsen coordinates are defined in the same way as the Fenchel-Nielsen parameters associated to geodesic pair of pants decomposition in the case of surfaces of finite type, but some care has to be taken regarding the existence of such a pair of pants decomposition in the infinite type case. In the paper \cite{ALPSS} we gave a necessary and sufficient condition on a hyperbolic structure on a surface of infinite type $S$ so that a topological pair of pants decomposition of $S$ can be made geodesic (see \cite[Theorem  4.5]{ALPSS}). We called this condition \emph{Nielsen-convexity}.  

Given a hyperbolic structure $H$ on $S$, to each homotopy class of closed geodesic $C_i\in \mathcal{P}$, we associate a {\it length parameter} and a {\it twist parameter}. The length parameter is the familiar quantity  $l_H(C_i)\in ]0,\infty[$; that is, it is the length of the $H$-geodesic in the homotopy class $C_i$.  
The twist parameter is defined only if $C_i$ is not the homotopy class of a boundary component of $S$, and it measures the relative twist amount along the geodesic in the class $C_i$ between the two generalized pairs of pants that have this geodesic in common (the two pairs of pants can be the same). The definition is the same as the one that is done in the case of surfaces of finite type. A precise definition of the twist parameters is contained in \cite[Theorem 4.6.23]{Thurston}. The twist amount per unit time along the (geodesic in the class) $C_i$ is chosen to be proportional (and not necessarily equal) to arclength along that curve,  and we make the convention, as in \cite{ALPSS}, that a complete positive Dehn twist along the curve $C_i$ changes the twist parameter by addition of $2\pi$. Thus, in some sense, the parameter $\theta_H(C_i)$ that we are using is an ``angle" parameter. 
 
   The {\it Fenchel-Nielsen parameters} of $H$ is the collection of pairs 
$\left((l_H(C_i),\theta_H(C_i))\right)_{i\in I}$, where it
 is understood that if $C_i$ is homotopic to a boundary component, 
 then there is no twist parameter associated to it, and instead of a pair
  $(l_H(C_i),\theta_H(C_i))$, we have a single parameter
$l_H(C_i)$.

If two hyperbolic structures on $S$ are equivalent, then their Fenchel-Nielsen parameters are the same.

Given two hyperbolic metrics $H$ and $H'$ on $S$, we define their {\it Fenchel-Nielsen distance} with respect to $\mathcal{P}$ as
   
\begin{equation}\label{def:FND}
{d_{FN}(H,H')=\sup_{i\in I} \max\left(\left\vert \log \frac{l_H(C_i)}{l_{H'}(C_i)}\right\vert, \vert l_H(C_i)\theta_H(C_i)-l_{H'}(C_i)\theta_{H'}(C_i)\vert \right) },
\end{equation}
again with the convention that if $C_i$ is the homotopy class of a boundary component of $S$, then there is no twist parameter to be considered.  
  
Given two hyperbolic structures $H$ and $H'$ on $S$, we say that they are  {\it Fenchel-Nielsen bounded} (relatively to $\mathcal{P}$) if their Fenchel-distance is finite. Fenchel-Nielsen boundedness is an equivalence relation.

Let $H_0$ be a homotopy class of a hyperbolic structure on $S$, which we shall consider as a base element of Teichm\"uller space.
We consider the collection of homotopy classes of hyperbolic structures $H$ which are Fenchel-Nielsen bounded from $H_0$ and with respect $\mathcal{P}$. Given two such homotopy classes of hyperbolic structures $H$ and $H'$, we write $H\sim H'$ if there exists an isometry from $(S,H)$ to  $(S,H')$ which is homotopic to the identity. The relation $\sim$ is an equivalence relation on the set of Fenchel-Nielsen bounded homotopy classes of hyperbolic surfaces $H$ based at $H_0$.

\begin{definition}[Fenchel-Nielsen Teichm\"uller space] The {\it Fenchel-Nielsen Teich\-m\"uller space} with respect to $\mathcal{P}$ and with basepoint $H_0$, denoted by $\mathcal{T}_{FN}(H_0)$,  is the space of $\sim$-equivalence classes of hyperbolic structures which are  Fenchel-Nielsen bounded relative to $H_0$ and $\mathcal{P}$.

 The function $d_{FN}$ defined in (\ref{def:FND}) is clearly a distance function on $\mathcal{T}_{FN}(H_0)$.  The {\it basepoint} of this Teichm\"uller space is the homotopy class $H_0$.
\end{definition}

 We shall call the distance $d_{FN}$ on $\mathcal{T}_{FN}(H_0)$ the {\it Fenchel-Nielsen distance} relative to the pair of pants decomposition $\mathcal{P}$.
The map 
$$
\mathcal{T}_{FN}(H_0) \ni H \mapsto {(\log(l_H(C_i)), l_H(C_i)\theta_H(C_i))}_{i \in I}\in \ell^\infty $$
 is an isometric bijection between $\mathcal{T}_{FN}(H_0)$ and the sequence space $l^\infty$. It follows from general properties of $l^{\infty}$-norms that the Fenchel-Nielsen distance on $\mathcal{T}_{FN}(H_0)$ is complete.

We prove in the next two propositions that we have in general $\mathcal{T}_{ls}(H_0)\not\subset\mathcal{T}_{FN}(H_0)$ and that the lengh-spectrum distance and Fenchel-Nielsen distance might behave very differently.

\begin{proposition}\label{prop:non-inc}
Let $H_0$ be a hyperbolic structure on $S$, such that there exists a
sequence of homotopy classes of disjoint essential simple closed curves on
$S$ whose lengths tends to 0. Then there exists an element $H$ in $\mathcal{T}_{ls}(H_0)$ with $H\not\in
\mathcal{T}_{FN}(H_0)$.
  \end{proposition}

\begin{proof}  Without loss of generality, we may assume that there exists a
sequence
 $\alpha_n, n=1,2 \ldots$, of homotopy classes of disjoint essential
closed curves on $S$
 whose lengths satisfy  $l_{H_{0}}(\alpha_n)=\epsilon_n$ with $e^{-(n+1)^2}<\epsilon_n<e^{-n^2}$.

Let $$ t_n= [\frac{\log |\log \epsilon_n |}{\epsilon_n}]+1, \ n=1,
2, \ldots $$ 
where $[r]$ denotes the integral part of the real number $r$.

 For each $n=1, 2, \ldots$, let $\tau_n$ be the $t_n$-th power of the positive Dehn
twist about $\alpha_n$. We take all the positive Dehn twists to be
supported on disjoint annuli, we let $T$ be the infinite composition $\tau_1 \circ \tau_2 \circ \ldots$, and we set $H=T(H_0)$.
 Then for every $n=1,2, \cdots$, we have, from the definition of the Fenchel-Nielsen distance,

  \begin{eqnarray*}
d_{FN}({H_0}, H)) &\geq  & 2\pi  t_n l_{H_{0}}(\alpha_n) \\
&= & 2\pi  t_n \epsilon_n\\
& \geq & 2\pi \log |\log \epsilon_n|.
\end{eqnarray*}

Since
$\lim_{n\to\infty}\epsilon_n=0$, we obtain $d_{FN}({H_0}, H)=\infty$.

The proof that $d_{ls}({H_0},H)<\infty$ is given in \cite[Proposition 4.7]{LP}.
\end{proof}

\begin{proposition}\label{prop:non-inc2}
Let ${H_0}$ be a hyperbolic structure on $X$, such that there exists a sequence of homotopy classes of disjoint essential simple closed curves in
$X$ whose lengths tend to 0. Then there exists a sequence of elements
$x_i$, $i=0,1, \cdots$  in $\mathcal{T}_{ls}({H_0})\cap \mathcal{T}_{FN}({H_0})$ such that $\lim_{n\to\infty}d_{ls}(x_n,
H_0)=0$, while $\lim_{n\to\infty}d_{FN}(x_n, H_0)=\infty$.
  \end{proposition}

\begin{proof} We consider a set $\alpha_n$ of homotopy classes of disjoint simple closed curves satisfying the same properties as in the proof of Proposition \ref{prop:non-inc}. We take the same definition of $\epsilon_n$, of $t_n$ and of the multiple Dehn twists $\tau_n$ supported on disjoint annuli. 
 Then for each $n=1,2, \cdots$,
 we have

  \begin{eqnarray*}
d_{FN}(H_0, \tau_n(H_0)) &\geq  & 2\pi  t_n l_{H_{0}}(\alpha_n) \\
&= & 2\pi  t_n \epsilon_n\\
& \geq & 2\pi  \log |\log \epsilon_n|.
\end{eqnarray*}
  
Since the above inequality is valid for any $n\ge 1$ and since
$\lim_{n\to\infty}\epsilon_n=0$,  we have $\lim_{n\to\infty}d_{FN}(H_0,
\tau_n(H_0))=\infty$.

Next we show that
\[\label{eq:lss}\lim_{n\to\infty}d_{ls}({H_0}, \tau_n ({H_0}))=\lim_{n\to\infty} \log \sup_{\alpha\in\mathcal{S}(X)} \big\{
\frac{l_{{H_0}}(\tau_n(\alpha))}{l_{H_{0}}(\alpha)},
\frac{l_{H_{0}}(\alpha)}{l_{H_{0}}(\tau_n(\alpha))}\big\}=0.\]

The proof is adapted from the proof of Propositions 2.13 and 4.7 of \cite{LP}. 

 Let $\alpha$
be an arbitrary homotopy class of essential curves in $X.$ 

For $i\in I$, if $i(\alpha, \alpha_n)=0$, then $\alpha=\tau_n(\alpha)$ and $l_{H_{0}}
(\tau_n (\alpha))= l_{H_{0}} (\alpha)$.

Assume now that $i(\alpha, \alpha_n) \neq 0$. By the Collar Lemma (see \cite{Buser}),
on any hyperbolic surface $H$, any closed geodesic whose length
$\epsilon$ is sufficiently small has an embedded collar neighborhood
of width $\vert \log\epsilon\vert$. Thus, we can write,
 for all $n\geq 0$,
 \[ l_{H_{0}}(\tau_n(\alpha)) \geq i(\alpha,\alpha_n)\vert \log \epsilon_n\vert.\]
 
From the definition of a Dehn twist, we then have

\[ l_{H_{0}}(\alpha) \leq  l_{H_{0}}(\tau_n(\alpha)) + i(\alpha,\alpha_n)t_n\epsilon_n.\]
 
 Thus, we obtain
   \begin{eqnarray*}
\frac{l_{H_{0}}(\tau_n(\alpha))}{l_{H_{0}}(\alpha)}
&\leq & + \frac{i(\alpha,\alpha_n)t_n\epsilon_n}{l_{H_{0}}(\tau_n(\alpha))}\\
  \\&\leq&
 1+  \frac {\log|\log \epsilon_n|} {|\log
\epsilon_n|}
  \\&\leq& 1+ 2\frac{\log (n+1)}{n^2}
  \\&\leq & 1+ \frac 2 n.
\end{eqnarray*}
which is bounded independently of $\alpha$ and $n$. In the same way,
we can prove that 
$\displaystyle \frac{l_{H_{0}}(\alpha)}{l_{H_{0}}(\tau_n(\alpha))}\leq 1+ \frac 2 n.$  This gives $\lim_{n\to\infty}d_{ls}({H_0},
\tau_n({H_0}))=\lim_{n\to\infty}d_{ls}(x_n, H_0)=0$.

\end{proof}

\section{Completeness of the length spectrum metric}\label{s:completeness}

In this section, $H_0$ is a hyperbolic structure on $S$,  $(\mathcal{T}_{ls}(H_0), d_{ls})$ is the length-spectrum Teichm\"uller space based at this point, equipped with the length-spectrum distance, 
and
 $\mathcal{P}=\{C_i\}_{i \in I}$ is a hyperbolic pair of pants decomposition of $H_0$. For every hyperbolic structure $H$ on $S$, we denote by ${(l_H(C_i),\theta_H(C_i))}_{i \in I}$ its Fenchel-Nielsen coordinates with respect to $\mathcal{P}$.

\begin{lemma} \label{lemma:pointwise_FN}
Let $(x_n) \subset \mathcal{T}_{ls}(H_0)$ be a sequence converging to a point  $x$ in $\mathcal{T}_{ls}(H_0)$. Then for all $i\in I$ we have $l_{x_n}(C_i) \rightarrow l_x(C_i)$ and $\theta_{x_n}(C_i) \rightarrow \theta_x(C_i)$.  
\end{lemma}

        \bigskip
  \begin{figure}[!hbp]
\centering
\includegraphics[width=.40\linewidth]{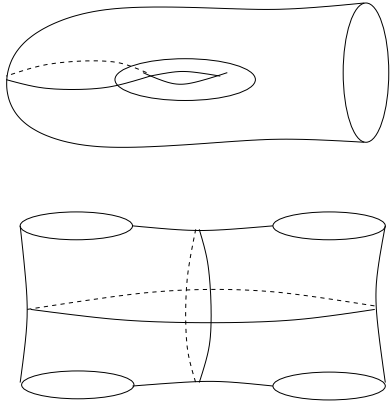}
\caption{\small {The curve $\beta_i$ used in the proof of Lemma \ref{lemma:pointwise_FN}  and of Lemma \ref{lemma:pointwise_FN_2}.
In each case, we have represented the simple closed curves $C_i$ and $\beta_i$.}}
\label{dual}
\end{figure}
\bigskip

\begin{proof}
From the definition of the length-spectrum distance, for every simple closed curve $\gamma \subset S$ we have that $l_{x_n}(\gamma) \rightarrow l_{x}(\gamma)$. In particular $l_{x_n}(C_i) \rightarrow l_x(C_i)$. For every curve $C_i$, we consider an essential simple closed curve $\beta_i$ which is not homotopic to it and intersecting it in a minimal number of points (which is two or one) and which is disjoint from $C_j$ for any $j\not= i$. (See Figure \ref{dual}.)
We let $\beta_i'$ be the image of  $\beta_i$ by the Dehn twist along $C_i$.  We also have $l_{x_n}(\beta_i) \rightarrow l_x(\beta_i)$. By the formulae in \cite{Okai}, the absolute value of the twist parameter along $C_i$ is a continuous function of the length of the curves $C_i, \beta_i$, and of the other curves $C_j$ at the boundaries of the pair of pants containing $C_i$. Hence $|\theta_{x_n}(C_i)| \rightarrow |\theta_x(C_i)|$. If $\theta_x(C_i) = 0$ we are done, otherwise note that by the same argument, using $\beta_i'$ instead of $\beta_i$, we have $|\theta_{x_n}(C_i) + 2 \pi| \rightarrow | \theta_x(C_i) + 2 \pi|$, hence for $n$ large enough, $\theta_{x_n}(C_i)$ and  $\theta_x(C_i)$ have the same sign.
\end{proof}

One may ask whether the converse of this lemma is true, i.e. if $(x_n) \subset \mathcal{T}_{ls}(H_0)$ is some sequence, and if $x\in \mathcal{T}_{ls}(H_0)$ is  such that $l_{x_n}(C_i) \rightarrow l_x(C_i)$ and $\theta_{x_n}(C_i) \rightarrow \theta_x(C_i)$, then is it true that $(x_n) \rightarrow x$ in the lengh-spectrum metric ?

We prove a result of this kind under an additional hypothesis on $(x_n)$, see Lemma \ref{lemma:ls_convergence}.

\begin{lemma} \label{lemma:pointwise_FN_2}
Let $(x_n) \subset \mathcal{T}_{ls}(H_0)$ be a Cauchy sequence. Then there are numbers $l_i \in \mathbb{R}_{>0}$ and $\theta_i \in \mathbb{R}$ such that for all $i\in I$ we have $l_{x_n}(C_i) \rightarrow l_i$ and $\theta_{x_n}(C_i) \rightarrow \theta_i$.  
\end{lemma}
\begin{proof}
By the definition of the length-spectrum distance, for every simple closed curve $\gamma$ on $S$, the sequence $\log(l_{x_n}(\gamma))$ is a Cauchy sequence of real numbers. In particular, there exists a positive real number $l_\gamma$ such that $l_{x_n}(\gamma)\rightarrow l_\gamma$. In particular $l_{x_n}(C_i) \rightarrow l_{C_i} = l_i$. Consider the curves $\beta_i, \beta_i'$ as in Lemma \ref{lemma:pointwise_FN}. By using the formulae of \cite{Okai} as in Lemma \ref{lemma:pointwise_FN}, we can see that $|\theta_{x_n}(C_i)|$ converges to a non-negative real number. If this number is zero, we put $\theta_i = 0$, otherwise we  choose $\theta_i$ such that $|\theta_i|$ is that number. To choose the sign of $\theta_i$, we use the limit of the sequence $|\theta_{x_n}(C_i) + 2 \pi|$, the sign of $\theta_i$ being positive if this limit is greater that $l_i$, otherwise this sign being negative. With these choices we have $\theta_{x_n}(C_i) \rightarrow \theta_i$. 
\end{proof}

\begin{lemma} \label{lemma:length_convergence}
Let $(x_n) \subset \mathcal{T}_{ls}(H_0)$ be a sequence, and let $x\in \mathcal{T}_{ls}(H_0)$ be such that $l_{x_n}(C_i) \rightarrow l_x(C_i)$ and $\theta_{x_n}(C_i) \rightarrow \theta_x(C_i)$.  Then for every element $\gamma$ in $\mathcal{S}$, we have $l_{x_n}(\gamma) \rightarrow l_x(\gamma)$.
\end{lemma}
\begin{proof}
The closed curve $\gamma$ is compact, hence it is contained in a subsurface $S'$ of $S$ that is the union of finitely many pairs of pants of the decomposition $\mathcal{P}$. Choose representatives in the equivalence classes of the structures $x_n$ and $x$ such that the boundary curves of $S'$ are geodesics. Consider the restrictions $x_n'$ and $x'$ of our hyperbolic structures to $S'$. On this finite type subsurface there are only finitely many Fenchel-Nielsen coordinates, hence the surfaces $x_n'$ and $x'$ are upper bounded and $d_{FN}(x_n',x') \rightarrow 0$. By the result in \cite{ALPSS}, we have $d_{qc}(x_n',x') \rightarrow 0$, which, by Wolpert's Inequality, implies $d_{ls}(x_n',x') \rightarrow 0$. In particular $l_{x_n}(\gamma) \rightarrow l_x(\gamma)$.   
\end{proof}

\begin{lemma}  \label{lemma:ls_convergence}
Let $(x_n) \subset (\mathcal{T}_{ls}(H_0), d_{ls})$ be a Cauchy sequence, and let $x\in \mathcal{T}_{ls}(H_0)$ be such that $l_{x_n}(C_i) \rightarrow l_x(C_i)$ and $\theta_{x_n}(C_i) \rightarrow \theta_x(C_i)$.  Then $d_{ls}(x_n,x) \rightarrow 0$.
\end{lemma}
\begin{proof}
By hypothesis, $(x_n)$ is a Cauchy sequence; that is: 
$$\forall \epsilon >0, \exists N: \forall n,m > N, d_{ls}(x_n,x_m) < \epsilon$$
Take an element $\gamma$ of $\mathcal{S}$. From the above property, we have, $\forall n,m > N, $
$$\left|\log\frac{l_{x_n}(\gamma)}{l_{x_m}(\gamma)}\right| < \epsilon$$
By Lemma \ref{lemma:length_convergence} we have $l_{x_m}(\gamma) \rightarrow l_x(\gamma)$, hence $\forall n > N, $ 
$$\left|\log\frac{l_{x_n}(\gamma)}{l_{x}(\gamma)}\right| \leq \epsilon$$
Taking the supremum over all $\gamma$ in $\mathcal{S}$, we have

$$\forall \epsilon >0, \exists N: \forall n > N, d_{ls}(x_n,x) \leq \epsilon;$$
that is, $x_n\to x$.
\end{proof}

\begin{theorem}\label{Theorem:complete}
For any hyperbolic metric $H_0$ on $S$, the metric space $(\mathcal{T}_{ls}(H_0), d_{ls})$ is complete.
\end{theorem}
\begin{proof}
This is a direct corollary of Lemmas \ref{lemma:pointwise_FN_2} and \ref{lemma:ls_convergence}. Take a Cauchy sequence $(x_n)$ in $\mathcal{T}_{ls}(X)$. By Lemma \ref{lemma:pointwise_FN_2}, we can find the limits of length and twist parameters  $(l_i, \theta_i)$ of $C_i$. Use these numbers to construct a marked hyperbolic surface with Fenchel-Nielsen coordinates $(l_i, \theta_i)$. By Lemma \ref{lemma:ls_convergence}, the sequence $x_n$ converges to this marked hyperbolic surface. Hence every Cauchy sequence has a limit.
\end{proof}

\begin{remarks}
1) Theorem \ref{Theorem:complete} answers Question 2.22 of \cite{LP}, which asks for necessary and sufficient condition for a hyperbolic structure $S$ on a  Riemann surface of infinite topological type under which the length-spectrum Teichm\"uller space $(\mathcal{T}_{ls}(H_0), d_{ls})$ is complete.

2) The proof of Theorem \ref{Theorem:complete} also works for surfaces of finite type. For such surfaces, the result was already known, see \cite [Theorem 2.25]{LP}.

\end{remarks}

\begin{question} we have the inclusion $\mathcal{T}_{qc}(X) \subset \mathcal{T}_{ls}(X)$, and we proved that this is not always an equality. Is it true that $\mathcal{T}_{qc}(X)$ is dense in $\mathcal{T}_{ls}(X)$ ?  If this were true, $(\mathcal{T}_{ls}(X),d_{ls})$ would be the metric completion of $(\mathcal{T}_{qc}(X),d_{ls})$.
 \end{question}

\section{More examples}
In this section, we give examples of a  hyperbolic structure $H_0$  such that the restriction of the length-spectrum metric $d_{ls}$
to the Teichm\"uller space $\mathcal{T}_{qs}(H)$  is not complete. Of course, the hyperbolic structures does not satisfy Shiga's condition (\ref{LM}).

The first example is an adaptation of an example that was given in \cite{LSW}.
\begin{example}\label{example:non-lower}
Let $H_0$ be a hyperbolic surface with a pants decomposition $\mathcal{P}=\{C_i\ \vert \ i\in I\}$, such that for some
subsequence of $C_{i_k}$ contained in the interior of $H_0$, $l_{H_{0}}(C_{i_k})=\epsilon_k\to 0$. For each $n=1,2,\ldots$, let $H_n$
be the hyperbolic surface obtained
by a positive multi-Dehn twist of $H_0$ along $C_{i_n}$ of order $t_n=[\log|\log \epsilon_n|]$. Note that
$t_n\to \infty$ as $n\to\infty$ but $\frac{t_n}{log \epsilon_n}\to 0$. We show that
$$d_{ls}(H_n,H_0)\to 0, \ \mbox{ while } \ d_{qc}(H_n,H_0)\to \infty.$$
\end{example}

Let us first show that $d_{ls}(H_n,H_0)\to 0$.
Recall that the length  spectrum metric is defined by
$$d_{ls}(H_n,H_0) = \max \log
\sup_\gamma\frac{l_{H_n}(\gamma)}{l_{H_0}(\gamma)},\log
\sup_\gamma\frac{l_{H}(\gamma)}{l_{H_n}(\gamma)}\},$$
where the supremum is taken over all essential simple closed curves $\gamma$ on $S$.

If for some $k$ a simple closed curve $\gamma$ does not intersect $C_{i_k}$, then
the hyperbolic length of $\gamma$ is invariant under the twist along $C_{i_k}$.
If $\gamma$ intersects $C_{i_k}$, we have
$$l_{H_k}(\gamma)-i(\gamma,C_{i_k})t_k \leq l_{H_{0}}(\gamma)\leq l_{H_k}(\gamma)+i(\gamma,C_i)t_k.$$
As a result,
$$d_{ls}(H_k,H_0)
= \max \{\log
\sup_{i(\gamma, C_{i_k})\ne 0}\frac{l_{H_k}(\gamma)}{l_{H}(\gamma)},\log
\sup_{ i(\gamma, C_{i_k})\ne 0}\frac{l_{H_0}(\gamma)}{l_{H_k}(\gamma)}\}.
$$

We have
$$\log\frac{l_{H_k}(\gamma)}{l_{H_0}(\gamma)} \leq \log \frac{ l_{H_0}(\gamma)+i(\gamma,C_{i_k})t_k}{l_{H_0}(\gamma)}
=\log (1+\frac{i(\gamma,C_{i_k})t_k}{l_{H_0}(\gamma)})
\leq \frac{i(\gamma,C_{i_k})t_k}{l_{H_0}(\gamma)},$$
and similarly,
$$\log\frac{l_{{H_0}}(\gamma)}{l_{H_k}(\gamma)} \leq \log \frac{l_{H_0}(\gamma)} { l_{H_0}(\gamma)-i(\gamma,C_{i_k})t_k}\leq \frac{i(\gamma,C_{i_k})t_k}{l_{H_0}(\gamma)}.$$
Thus, we have
\begin{equation}\label{inequ:ls}
d_{ls}(H_k,H_0) \le
\sup_{i(\gamma, C_{i_k}) \ne 0}\frac{i(\gamma, C_{i_k})t_k}{l_{H_0}(\gamma)}.
\end{equation}

Note that as $l_{H_0}(C_{i_k})\rightarrow 0$, for any $i(\gamma,C_{i_k})\neq 0$, $l_{H_0}(\gamma)$ tends to infinity. In particular, if $l_{H_0}(C_{i_k})=\epsilon_k$,
then it follows from the Collar Lemma (see \cite{Buser}) that $l_{H_0}(\gamma)$ is bigger than $i(\gamma,C_{i_{k}})|\log \epsilon_k|$, up to a multiplicative constant.

Thus, we assume that every $\epsilon_k$ is less than some fixed constant $M>0$. Then there is a constant $C$ depending on $M$, such that,  $l_{H_0}(\gamma)$ is larger than $C i(\gamma, C_{i_k}) |\log \epsilon_k|$, as follows from the Collar Lemma \cite{Buser}. This lemma says that, for each simple closed geodesic with length $\ell$ less than $M$, there is a collar neighborhood of width larger than 
$w$, where $w$ is given by 
$$\sinh w= 1/ \sinh(\ell/2).$$
A simple computation shows that there is a constant $C$ depending on $M$ such that $w$ is larger than 
$C |\log \ell|$.
Since any simple closed curve $\gamma$ which intersects with $C_{i_k}$ should cross the collar neighborhood for $ i(\gamma, C_{i_k}) $ times, its hyperbolic length should be larger than $C i(\gamma, C_{i_k}) |\log \epsilon_k|$.

 As a result, the right hand side of Inequality $(\ref{inequ:ls})$ tends to $0$ as $k\to \infty$.
Thus we have
$d_{ls}(H_k,{H_0})\to 0. $

The proof of  $d_{qc}(H_k,H_0)\to \infty$ is given by Lemma 7.2 in \cite{ALPSS}.

For a generalization, see Theorem 7.6 in \cite{ALPSS}, and Theorem \ref{thm:multi} below.

\begin{example}\label{example:non-upper}
Let $H_0$ be a hyperbolic surface with a hyperbolic pants decomposition $\mathcal{P}=\{C_i\}$ such that for some
subsequence of $C_{i_k}$ contained in the interior of $H_0$ satisfies

(i) $l_{H_0}(C_{i_k})=a_k\to \infty$.

(ii) For any geodesic arc $\alpha$ connecting two points (not necessary distinct) on $C_{i_k}$, but $\alpha \subsetneq C_{i_k}$,
$\alpha$ has hyperbolic length $l_{H_0}(\alpha)> ka_k$. \\

Let $H_k$ be the hyperbolic surface obtained by positive Dehn twist of $H_0$ along $C_{i_k}$. Then for any simple
closed curve $\gamma$ such that $i(C_{i_k},\gamma) \not= 0$, $l_{H_0}(\gamma)$ is bigger than $k i(\gamma,C_{i_k})a_k$
and $l_{H_0}(\gamma)-i(\gamma,C_{i_k})a_k\leq l_{H_k}(\gamma)\leq l_{H_0}(\gamma)-i(\gamma,C_{i_k})a_k$. The same arguments with
the above example show that on the Teichm\"uller space $\mathcal{T}(H_0)_{qc}$, $d_{ls}(H_k,H_0)\to 0$ while $d_{qc}(H_k,H_0)\to \infty$.
\begin{remark}
Concrete examples satisfying the conditions (i) and (ii) in Example \ref{example:non-upper} are constructed
 by Shiga \cite{Shiga} and Matsuzaki \cite{Matsuzaki}.  Both of these examples also satisfy the condition that
the number of simple closed geodesics on $H_0$ whose lengths are uniformly
bounded from above is finite, therefore the hyperbolic structures are different from those of Example \ref{example:non-lower}. 
In the  example of Shiga \cite{Shiga}, the Riemann surface induced by $H_0$ is not complete. Matsuzaki \cite{Matsuzaki} has
refined Shiga's construction to give a complete Riemann surface $H_0$ and then he showed
that for such an $H_0$, the Teichm\"uller modular group $\mathrm{Mod}(H_0)$ has only a countable number
of elements.
\end{remark}
\end{example}
\begin{remark}
The above two examples show that there exist hyperbolic surfaces $H_0$ of infinite topological type, 
which do not satisfy Shiga's condition, such that $d_{qc}$ and $d_{ls}$
are
not topologically equivalent on $\mathcal{T}_{qc}(H_0)$.
\end{remark}

We now give a proof of a theorem due to Shiga \cite{Shiga}, which is different from the one given by Shiga.

\begin{theorem}\label{th:non-complete}
There exist surfaces $S$ of infinite topological type and hyperbolic structures on such surfaces such that the length-spectrum metric is not complete on $(\mathcal{T}_{qc}(H_0),d_{ls})$.
\end{theorem}
\begin{proof}
Examples of such hyperbolic structures are  those given in Examples \ref{example:non-lower} and \ref{example:non-upper} above.
We shall prove the required property for the structures in Example \ref{example:non-lower}. The proof for the structures given in Example \ref{example:non-upper} is similar.

 We consider the surface of Example \ref{example:non-lower} and we construct a Cauchy sequence in $(\mathcal{T}(H_0),d_{ls})$ that does not have a limit.    

Recall that  $H_0$ is a hyperbolic surface with a pants decomposition $\mathcal{P}=\{C_i\}$, such that for some
subsequence of $C_{i_k}$ contained in the interior of $H_0$, $l_{H_0}(C_{i_k})=\epsilon_k\to 0$, and that $t_k=[\log|\log \epsilon_k|]$.
Let $H_1$ be the hyperbolic surface obtained from $H_0$
by the  positive multiple Dehn twist of order $t_1$ along $C_{i_1}$. More generally, for all $k\geq 0$, let $H_{k}$ be the hyperbolic surface obtained from $H_{k-1}$
by the positive multiple Dehn twist  of order $t_k$ along $C_{i_k}$. Then as in the proof of Example \ref{example:non-lower},
we can show that $d_{ls}(H_m,H_n)\to 0$ as $m,n\to \infty.$ As a result, $(H_k)$ is Cauchy sequence in $(\mathcal{T}(H_0),d_{ls})$.
We prove that $(H_k)$ has no limit in $(\mathcal{T}_{qc}(H_0),d_{ls})$, by contradiction. Suppose there
is a hyperbolic surface $H\in \mathcal{T}_{qc}(H_0)$, such that $d_{ls}(H_k,H)\to 0$. 
Consider the Fenchel-Nielsen coordinates  determined by $H_0$ and $\mathcal{P}$.
From the construction of the sequence $(H_k)$, the Fenchel-Nielsen coordinates of $H$ are 
$\{(l_{H}(C_i),\theta_H(C_i))\}$, where $l_H(C_i)=l_{H_0}(C_i)$, $\theta_{H}(C_{i_k})$, with $\theta_{H}(C_{i_k})=2\pi \theta_k$ and $\theta_{H}(C_j)=0$ when $j\neq i_k$. 
We claim that $d_{qc}(H_0,H)=\infty$. As a result,
$H$ does not belong to $\mathcal{T}_{qc}(H_0)$, which is contradicted by the assumption. 

 To show that $d_{qc}(H_0,H)=\infty$, we use the following theorem  \cite[Theorem 7.6]{ALPSS}:
\begin{theorem}\label{thm:multi}
Let $H_0$ be a hyperbolic surface with a hyperbolic pair of pants decomposition $P=\{C_i\}$, and assume that there exists a positive constant $L_0$ such that $l_{H_0}(C_{i_k})\leq L$ for all
$k=1,2, \ldots $. Let $C_{i_{k}}$, $k=1,2,\ldots$ be a subsequence of $(C_i)$, and let $t=(t_k)$,  $k=1,2,\ldots$ be a sequence of positive real numbers. Let $H_t$ be the 
hyperbolic metric obtained by a Fenchel-Nielsen multi-twist along $C_{i_k}$, of distance $t_i$ measured on $C_{i_k}$, for each $k$.
Then if $d_{qc}(H_0,H_t)< M$, we have
$$\sup_k|t_k| \leq C d_{qc}(H,H_t)$$
where $C$ is a positive constant depending on $L$ and $M$.
\end{theorem}
It follows from the above theorem that if $d_{qc}(H_0, H)$ is finite, than $t_k$ is uniformly bounded, which 
contradicts the fact that $t_k=[\log|\log \epsilon_k|]\to \infty$.

\end{proof}

\begin{remark} Shiga's examples of hyperbolic structures are not complete (as metric spaces) whereas in  our examples they are complete. To see this,
note that since the geodesic length of each curve in the pairs of pants decomposition that we use is bounded uniformly from above, it follows that  any closed ball of radius 1 on the surface is contained in a finite number of pairs of pants of the given decomposition, and therefore it is compact. Thus, by the theorem of Hopf-Rinow, the metric is complete (see \cite[Lemma 4.7]{ALPSS}).
\end{remark}

 \end{document}